\documentclass[11pt]{article}

\usepackage{amsmath,amssymb,amsthm,mathtools}
\usepackage[margin=1in]{geometry}
\usepackage{enumitem}
\usepackage{booktabs}
\usepackage[colorlinks=true,linkcolor=blue,citecolor=blue,urlcolor=blue]{hyperref}

\newtheorem{theorem}{Theorem}[section]
\newtheorem{proposition}[theorem]{Proposition}
\newtheorem{corollary}[theorem]{Corollary}
\newtheorem{lemma}[theorem]{Lemma}
\newtheorem{definition}[theorem]{Definition}
\newtheorem{remark}[theorem]{Remark}

\newcommand{\R}{\mathbb R}
\newcommand{\C}{\mathbb C}
\newcommand{\HH}{\mathbb H}
\newcommand{\D}{\mathbb D}
\newcommand{\Harm}{\mathcal H}
\newcommand{\Herm}{\operatorname{Herm}}
\newcommand{\Sym}{\operatorname{Sym}}
\newcommand{\End}{\operatorname{End}}
\newcommand{\tr}{\operatorname{tr}}
\newcommand{\rank}{\operatorname{rank}}
\newcommand{\Rea}{\operatorname{Re}}

\newcommand{\Pfour}{\mathcal P_4}
\newcommand{\Gfour}{\mathcal G_4}

\title{Quadratic-Defect Completions of Spherical $2$-Design Orbits}
\author{Kuan-Cheng Chien and Ming-Hsuan Kang}
\date{}

\begin{document}
\maketitle

\begin{abstract}
We study how spherical $2$-designs arising from finite group orbits can be
completed to spherical $4$-designs by adjoining further orbits, allowing
weights in the general theory.  For an
irreducible real orthogonal $G$-module $W$ with
$\D=\End_G(W)\in\{\R,\C,\HH\}$, we consider the multiplicity-two
representation $W\oplus W$ and retain the failure of the $2$-design equation
$M^*M=\frac12I_2$ as a quadratic defect.  When the invariant quartics are
determined by the Hermitian Gram matrix, the fourth-moment problem reduces
to a mean and covariance condition on these defects.  This yields a sharp
lower bound for the total weight of the correction orbits; at equality,
their normalized defects form a weighted spherical $2$-design in the
associated defect space.  The quartic condition holds for the multiqubit
Clifford groups in every dimension $r\ge1$, giving an unbounded-dimensional
family with a fixed three-dimensional defect space; among equality cases
using the minimum number of correction orbits, the defect
geometry is always a regular tetrahedron.  As a complementary unweighted
example, we construct a $378$-point $W(E_6)$-invariant spherical
$4$-design in $S^{11}$ and prove that it is sharp among unweighted invariant
completions containing a spherical $2$-design orbit.
\end{abstract}

\section{Introduction}

A spherical $t$-design is a finite configuration on the sphere whose averages
agree with the uniform spherical measure for every polynomial of degree at
most $t$ \cite{DelsarteGoethalsSeidel1977}.  Finite group orbits are a
natural source of such designs: $G$-invariance reduces the defining moment
conditions to a small collection of invariant polynomials, a reduction we
used in \cite{ChienKang2025} to classify the spherical $2$-design orbits of
finite groups.  The present paper asks whether such an orbit can serve as
the building block of a design of higher strength---specifically, whether
it can be completed to a spherical $4$-design by adjoining further orbits,
allowing orbit weights in the general theory.

Our approach is to retain, rather than solve, the equations behind that
$2$-design criterion.  Each point of the relevant multiplicity-two
representation has an associated Hermitian Gram matrix, and the $2$-design
condition is exactly that this matrix be scalar; we track the failure of
this equation as a \emph{quadratic defect}, so that $2$-design orbits are
precisely the orbits of zero defect.  Working through degree four, and using
antipodalization to remove any odd contribution, we show that whenever a
single explicit character condition holds---that every invariant quartic is
already determined by the Gram matrix---the fourth-moment equations for a
union of orbits collapse to a vanishing mean and an isotropic covariance
condition on the defects, in a Euclidean space of dimension at most $5$.  A
$2$-design orbit, having zero defect, then serves as a base configuration,
and the missing fourth-moment information must come from \emph{correction
orbits} of nonzero defect.

This reduction yields a sharp lower bound on the total weight the correction
orbits must carry, with equality exactly when their normalized defects form
a weighted spherical $2$-design on an associated projective line.  The
underlying character condition holds well beyond isolated examples: it holds
for the dihedral groups $I_2(m)$ with $m=3$ or $m\ge5$, and for the multiqubit Clifford groups
\cite{Webb2016,Zhu2017} in every dimension $r\ge1$.  The latter give an
unbounded-dimensional family with a fixed three-dimensional defect space;
when equality in the correction-weight bound is attained with the minimum
possible number of correction orbits, those defects form the same regular
tetrahedron in every dimension.

As a complementary unweighted example, we construct a $378$-point spherical
$4$-design invariant under the Weyl group $W(E_6)$, completing a
$270$-point $2$-design orbit with four $27$-point correction orbits.  We
prove that $378$ is best possible among unweighted $W(E_6)$-invariant
$4$-designs containing a constituent orbit that is itself a spherical
$2$-design; at equality the orbit-size pattern is necessarily
$270+4\cdot27$, the correction orbits have rank-one Gram matrices, and their
normalized defects form a square.

This viewpoint is complementary to constructions of higher-order designs as
unions of lower-order ones, such as that of Mohammadpour and Waldron
\cite{MohammadpourWaldron2019}.  Here the correction constituents need not
be spherical designs in the ambient sphere: they are selected by their
nonzero quadratic defects.  The defect formulation additionally yields an
extremal correction-weight bound and a geometric description of its equality
case.  A recent Gramian characterization of spherical designs due to Waldron
\cite{WaldronGramian2025} instead uses the Gram matrix of the entire point
configuration and associated potentials.  The Gram matrix here has a different
role: it is the $2\times2$ multiplicity-space Gram matrix attached to a single
orbit representative, and its traceless part records the obstruction to the
orbit $2$-design equation and controls completion by further orbits.

The condition $\Pfour(G,W)=0$ thus isolates the rigid regime in which the full
fourth-order obstruction is already captured by the quadratic Gram defect.
It is natural to ask how this completion problem changes when primitive
quartic invariants are present.

Sections~\ref{sec:prelim}--\ref{sec:quartic} set up the quadratic defect and
its quartic-obstruction criterion; Section~\ref{sec:completion} proves the
completion theorem; Sections~\ref{sec:optimal}--\ref{sec:unweighted}
establish the sharp correction bound, its equality case, and examples;
Section~\ref{sec:E6} treats the $E_6$ construction.

\section{Preliminaries and the quadratic defect}
\label{sec:prelim}

Let $G$ be a finite group acting orthogonally on a real Euclidean space $V$.
For a finite subset $X\subset S(V)$, write
$$
\mu_X=\frac1{|X|}\sum_{x\in X}\delta_x.
$$
The set $X$ is a spherical $t$-design if
$$
\int f\,d\mu_X=\int_{S(V)}f\,d\sigma
$$
for every real polynomial $f$ of degree at most $t$, where $\sigma$ is
normalized spherical measure.  A finitely supported probability measure
satisfying the same identities will be called a \emph{weighted spherical
$t$-design}; an ordinary finite spherical design corresponds to equal
weights.

Write $\Harm_k(V)$ for the space of real homogeneous harmonic polynomials of
degree $k$ on $V$.  If a measure is $G$-invariant, Reynolds averaging shows
that it is enough to test $G$-invariant polynomials.  We will use this only
in degrees at most four.

Let $W$ be a nontrivial irreducible real orthogonal $G$-module.  By Schur's
lemma,
$$
\D:=\End_G(W)
$$
is one of $\R,\C,\HH$.  We use a right-module convention throughout.  In
quaternionic type this means replacing the natural left action of
$\End_G(W)$ by the equivalent right action through
$\End_G(W)^{\mathrm{op}}\simeq\HH$; for $\R$ and $\C$ this distinction is
immaterial.  Put
$$
d:=\dim_{\R}\D\in\{1,2,4\},
\qquad
n:=\dim_{\D}W,
\qquad
q:=\dim_{\R}W=dn.
$$
Choose a $G$-invariant $\D$-Hermitian form on $W$ whose real part is the given
real inner product.

We now pass to
$$
V=W\oplus W.
$$
For $x=(x_1,x_2)\in V$, set
$$
M_x=[x_1\ x_2]
$$
and
$$
H(x):=M_x^*M_x\in\Herm_2(\D).
$$
We henceforth identify $x$ with its two-column matrix $M_x$ whenever
convenient; in particular, notation such as $GM_0$ denotes the orbit of the
point represented by $M_0$.
If $x\in S(V)$, then
$$
H(x)\ge0,
\qquad
\tr H(x)=1.
$$

Let
$$
E_{\D}:=\Herm_2(\D)_0
$$
be the traceless Hermitian matrices, with inner product
$$
\langle B,C\rangle_E:=\Rea\tr(BC).
$$
Then
$$
r_{\D}:=\dim_{\R}E_{\D}=d+1=
\begin{cases}
2,&\D=\R,\\
3,&\D=\C,\\
5,&\D=\HH.
\end{cases}
$$

\begin{definition}
For $x\in S(V)$, the \emph{quadratic defect} of $x$ is
$$
Z(x):=H(x)-\frac12I_2\in E_{\D}.
$$
\end{definition}

Since $G$ acts simultaneously on the two columns of $M_x$ and preserves the
chosen $\D$-Hermitian form,
$$
H(gx)=H(x),\qquad Z(gx)=Z(x)\qquad(g\in G).
$$
Thus the defect is naturally attached to an orbit.

\begin{proposition}[Multiplicity-two $2$-design criterion]
\label{prop:2design}
For $x\in S(V)$, the following are equivalent:
\begin{enumerate}[label=\textup{(\roman*)}]
\item $Gx$ is a spherical $2$-design;
\item $Z(x)=0$;
\item $H(x)=\frac12I_2$.
\end{enumerate}
\end{proposition}

\begin{proof}
Since $W$ is nontrivial and irreducible, the first moment of every orbit
vanishes.  Thus only the second moment is relevant.

A $G$-invariant real quadratic form on $W\oplus W$ is determined by a
Hermitian matrix $B\in\Herm_2(\D)$:
$$
q_B(x)=\Rea\tr(BH(x)).
$$
The spherical average of $q_B$ is $\tr(B)/2$, because the two copies of $W$
have equal dimension.  Hence $Gx$ has the correct second moment if and only if
$$
\Rea\tr(BH(x))=\frac{\tr B}{2}
$$
for every $B\in\Herm_2(\D)$.  By nondegeneracy of the trace pairing this is
equivalent to $H(x)=I_2/2$.
\end{proof}

\begin{remark}
Proposition~\ref{prop:2design} is the multiplicity-two specialization of the
general matrix criterion in \cite{ChienKang2025}.  We include the short proof
so that the present paper can be read independently of the full
classification.
\end{remark}

\section{Primitive quartics and a character criterion}
\label{sec:primitive}
\label{sec:quartic}

Let $\Sym^4(V^*)^G$ denote the space of real $G$-invariant homogeneous
quartics on $V$.  For $B\in\Herm_2(\D)$, define the real quadratic Gram
coordinate
$$
\gamma_B(x):=\Rea\tr\bigl(BH(x)\bigr).
$$
As $B$ varies, these are exactly the real quadratic functions carried by the
Hermitian Gram matrix; this formulation avoids treating a complex or
quaternionic off-diagonal entry as a real polynomial.  Let
$$
\Gfour(G,W)\subseteq\Sym^4(V^*)^G
$$
be the space spanned by the products $\gamma_B\gamma_C$ with
$B,C\in\Herm_2(\D)$.

\begin{definition}
The \emph{primitive quartic obstruction space} is
$$
\Pfour(G,W):=
\Sym^4(V^*)^G/\Gfour(G,W).
$$
We say that $W\oplus W$ has \emph{no primitive quartic obstruction} if
$\Pfour(G,W)=0$.
\end{definition}

Thus $\Pfour(G,W)=0$ means exactly that every invariant quartic of the pair
$(x_1,x_2)$ is determined by its Hermitian Gram matrix.  This condition is
character-theoretic.  Character and Molien-series methods for spherical
designs from group orbits are standard; see, for example,
\cite{MohammadpourWaldron2024}.

Set
$$
h_{\D}:=\dim_{\R}\Herm_2(\D)=d+2
=
\begin{cases}
3,&\D=\R,\\
4,&\D=\C,\\
6,&\D=\HH,
\end{cases}
$$
and
$$
c_{\D}:=\dim_{\R}\Sym^2(\Herm_2(\D))
=\binom{h_{\D}+1}{2}
=
\begin{cases}
6,&\D=\R,\\
10,&\D=\C,\\
21,&\D=\HH.
\end{cases}
$$

\begin{lemma}
\label{lem:gramind}
Assume $n=\dim_{\D}W\ge2$.  Then the quadratic products in the Gram
coordinates are linearly independent.  In particular,
$$
\dim\Gfour(G,W)=c_{\D}.
$$
\end{lemma}

\begin{proof}
Every positive definite matrix in $\Herm_2(\D)$ occurs as the Gram matrix of
two vectors in $\D^n$ when $n\ge2$.  Hence the image of the Gram map contains
the open cone of positive definite Hermitian matrices.  A quadratic polynomial
in the real Gram coordinates which vanishes identically on $V$ therefore vanishes on
an open subset of $\Herm_2(\D)$ and must be zero.
\end{proof}

Let $\chi$ be the real character of $W$.  The invariant Euclidean form
identifies $V$ with $V^*$, so their symmetric powers have the same
characters.  Since $V=W\oplus W$ has character $2\chi$, define
$$
m_4(W):=\dim\Sym^4(V^*)^G
=\langle \chi_{\Sym^4(V)},1_G\rangle.
$$

\begin{proposition}[Character criterion]
\label{prop:character}
Assume $n\ge2$.  Then
$$
\dim\Pfour(G,W)=m_4(W)-c_{\D}.
$$
In particular,
$$
\Pfour(G,W)=0
$$
if and only if
$$
m_4(W)=
\begin{cases}
6,&\D=\R,\\
10,&\D=\C,\\
21,&\D=\HH.
\end{cases}
$$

Moreover,
$$
m_4(W)
=
\frac1{|G|}\sum_{g\in G}
\left[
\frac23\chi(g)^4
+2\chi(g)^2\chi(g^2)
+\frac12\chi(g^2)^2
+\frac43\chi(g)\chi(g^3)
+\frac12\chi(g^4)
\right].
$$
\end{proposition}

\begin{proof}
The first assertion follows from Lemma~\ref{lem:gramind}, since
$\Gfour(G,W)\subseteq\Sym^4(V^*)^G$.  For the formula, use the standard
identity
$$
\chi_{\Sym^4 V}(g)
=
\frac1{24}\left(
\psi(g)^4+6\psi(g)^2\psi(g^2)+3\psi(g^2)^2
+8\psi(g)\psi(g^3)+6\psi(g^4)
\right)
$$
with $\psi=2\chi$, and average over $G$.
\end{proof}

For later use, the cubic invariant multiplicity is equally easy to test:
\begin{equation}
\label{eq:cubic-character}
m_3(W):=\dim\Sym^3((W\oplus W)^*)^G
=
\frac1{|G|}\sum_{g\in G}
\left[
\frac43\chi(g)^3
+2\chi(g)\chi(g^2)
+\frac23\chi(g^3)
\right].
\end{equation}
Thus the cubic condition, when one wants the original configuration rather
than its antipodalization, is also determined directly from the same character
data.  Since $V^G=0$, the harmonic decomposition
$$
\Sym^3(V^*)=\Harm_3(V)\oplus \|x\|^2\Harm_1(V)
$$
shows that
$$
m_3(W)=\dim \Harm_3(V)^G.
$$
Thus $m_3(W)=0$ is equivalent here to the absence of invariant harmonic
cubics.

\begin{remark}
The Frobenius--Schur type $\D$ is itself determined by character data.  Thus,
once the power maps in the character table are known, the applicability of
the defect completion theorem is completely character-theoretic.
\end{remark}

\subsection{Two useful families}

The character criterion shows that the quartic hypothesis is not an isolated
phenomenon.

\begin{proposition}[Dihedral family]
\label{prop:dihedral}
Let $W=\R^2$ be the standard reflection representation of the dihedral group
$I_2(m)$.  Then
$$
\Pfour(I_2(m),W)=0
$$
for
$$
m=3\quad\text{or}\quad m\ge5.
$$
For $m=4$ there are primitive quartic invariants.
\end{proposition}

\begin{proof}
After complexification, a rotation acts on $W_{\C}$ with weights $1$ and
$-1$.  Hence degree-four monomials on
$(W\oplus W)_{\C}$ have rotation weights
$$
-4,-2,0,2,4.
$$
If $m=3$ or $m\ge5$, invariance under the rotation subgroup forces weight
zero.  After imposing a reflection, the degree-four invariants are therefore
the same as the $O(2)$-invariant quartics of two vectors, which are generated
by the three real Gram entries.  For $m=4$, weight $\pm4$ is also invariant,
giving additional quartics.
\end{proof}

\begin{remark}
The case $I_2(3)$ has invariant cubics even though it has no primitive quartic
obstruction.  It therefore gives a simple example where the defect equations
solve the even fourth-moment problem but antipodalization is still needed to
obtain a spherical $4$-design.
\end{remark}

The next proposition draws on a second, standard notion also called a
\emph{design}: a finite group $G\le U(T)$ is a \emph{unitary $t$-design} if
its average of degree-$(t,t)$ polynomials in the matrix entries agrees with
the Haar average on $U(T)$ \cite{Webb2016,Zhu2017}.  This is a property of
the group's action on $T$, distinct from the spherical designs of
Section~\ref{sec:prelim}; here it is used only to verify the primitive
quartic condition $\Pfour(G,W)=0$, and the resulting spherical designs are
constructed only afterward, in Theorem~\ref{thm:completion}.

\begin{proposition}[Unitary-design family]
\label{prop:unitaryfamily}
Let $T$ be a complex vector space of dimension at least two and let
$G\le U(T)$ be a finite unitary $2$-design.  Assume that $G$ contains a
scalar $\zeta I$ with $\zeta^4\ne1$.  Let $W=T_{\R}$ be the underlying real
representation.  Then $\End_G(W)=\C$ and $\Pfour(G,W)=0$.
\end{proposition}

\begin{proof}
Write a real quartic in complex coordinates and decompose it into bidegrees
$(a,b)$ in the variables and their conjugates, where $a+b=4$.  Note
$\zeta^4\ne1$ already forces $\zeta^2\ne1$, since $\zeta^2=1$ would give
$\zeta^4=(\zeta^2)^2=1$.  Invariance under $\zeta I$ therefore forces $a=b=2$,
since the other possible bidegrees have $a-b=\pm2$ or $\pm4$.  Since $G$ is a
unitary $2$-design, its invariants of bidegree $(2,2)$ agree with those of
$U(T)$.  By the first fundamental theorem for the unitary group, these are
generated by Hermitian pairings of the two vector variables, hence by the
entries of $M^*M$ \cite{Weyl1939}.

It remains to determine the real commutant.  Since $|\zeta|=1$ and
$\zeta^4\ne1$, the scalar $\zeta$ is non-real.  If a real-linear operator
$A$ commutes with $G$, then in particular it commutes with multiplication by
$\zeta$.  Writing $\zeta=a+bi$ with $b\ne0$, it follows that $A$ commutes
with the complex structure
$$
iI=\frac{\zeta I-aI}{b}.
$$
Hence every real-linear element of $\End_G(W)$ is in fact complex-linear.
The unitary $1$-design consequence says that the complex-linear commutant is
the scalar algebra.  Therefore $\End_G(W)=\C$.
\end{proof}

\begin{remark}[Phase extension]
\label{rem:phaseextension}
The scalar hypothesis in Proposition~\ref{prop:unitaryfamily} can always be
arranged by a harmless central extension.  If $K\le U(T)$ is any finite
unitary $2$-design and $\zeta=e^{\pi i/4}$, then
$$
G=\langle K,\zeta I\rangle
$$
is again a finite unitary $2$-design: the added scalar phases cancel in the
balanced degree-$(2,2)$ moments.  Proposition~\ref{prop:unitaryfamily}
therefore applies to this phase extension.  The Clifford family below gives
a standard unbounded-dimensional realization in which such a phase is
already built into the chosen finite lift.
\end{remark}

To show that Proposition~\ref{prop:unitaryfamily} is not satisfied only by
finitely many sporadic examples, we now exhibit an infinite family, indexed
by the number of qubits, for which the hypothesis holds in every dimension.
Since the multiqubit Clifford group is less classical here than the
reflection groups used elsewhere in the paper, we spell out its
construction explicitly.  Put $T_r=(\C^2)^{\otimes r}$.  On one qubit let
$$
X=\begin{pmatrix}0&1\\1&0\end{pmatrix},
\qquad
Z=\begin{pmatrix}1&0\\0&-1\end{pmatrix}
$$
be the standard Pauli matrices (used only in this construction, and
unrelated to the point configuration $X$ or the defect $Z$ used elsewhere in
the paper), and let $X_j,Z_j$ act as $X,Z$ on the $j$th tensor factor and
trivially on the others.  The $r$-qubit Pauli group is
$$
\mathcal P_r=\langle iI,X_1,Z_1,\ldots,X_r,Z_r\rangle\le U(T_r).
$$
The projective Clifford group is the normalizer of $\mathcal P_r$ in
$U(T_r)$ modulo scalar phases.  We use the standard finite phase lift
$\widetilde{\mathrm{Cliff}}_r$, generated by the one-qubit gates
$$
\mathsf H=\frac1{\sqrt2}
\begin{pmatrix}1&1\\1&-1\end{pmatrix},
\qquad
S=\begin{pmatrix}1&0\\0&i\end{pmatrix},
$$
acting on individual tensor factors, together with the controlled-NOT gates
$$
\mathrm{CNOT}_{jk}|a,b\rangle=|a,a+b\pmod2\rangle
$$
on pairs of tensor factors and the scalar $e^{\pi i/4}I$.  This is a finite
unitary group; modulo scalar phases it is the usual Clifford group.  The
choice of finite lift is immaterial for unitary design moments, since scalar
phases cancel between matrix entries and their conjugates.

\begin{corollary}[Multiqubit Clifford family]
\label{cor:cliffordfamily}
For every $r\ge1$, let $W_r=(T_r)_{\R}$ be the underlying real
representation of $\widetilde{\mathrm{Cliff}}_r$.  Then
$\End_{\widetilde{\mathrm{Cliff}}_r}(W_r)=\C$ and
$\Pfour(\widetilde{\mathrm{Cliff}}_r,W_r)=0$.  In particular,
$$
\dim_{\R}W_r=2^{r+1}\longrightarrow\infty,
\qquad
\dim_{\R}E_{\C}=3.
$$
Thus the no-primitive-quartic hypothesis of
Theorem~\ref{thm:completion} occurs in an unbounded-dimensional family while
the associated defect space has fixed dimension three.
\end{corollary}

\begin{proof}
The multiqubit Clifford group is a unitary $3$-design for every $r\ge1$
\cite{Webb2016,Zhu2017}, hence in particular a unitary $2$-design.  The
chosen finite lift contains $e^{\pi i/4}I$, whose fourth power is $-1\ne1$.
Proposition~\ref{prop:unitaryfamily} therefore applies.  Since
$\dim_{\C}T_r=2^r$, the final dimension statements follow immediately.
\end{proof}

\section{The quadratic-defect completion theorem}
\label{sec:completion}

We now isolate the even moments.  For a probability measure $\mu$ on $S(V)$,
write
$$
\mu^{\pm}:=\frac12\bigl(\mu+(-I)_*\mu\bigr)
$$
for its antipodalization.  Antipodalization leaves all even moments unchanged
and kills all odd moments.

\begin{lemma}[Spherical covariance of the defect]
\label{lem:cov}
Let $\sigma$ be normalized spherical measure on $S(V)=S^{2q-1}$.  Then
$$
\int Z(x)\,d\sigma(x)=0
$$
and
$$
\int Z(x)\otimes Z(x)\,d\sigma(x)
=
\frac{1}{2(q+1)}I_{E_{\D}}.
$$
\end{lemma}

\begin{proof}
For $B\in E_{\D}$, define
$$
q_B(x):=\langle B,Z(x)\rangle_E.
$$
The corresponding real self-adjoint operator $A_B$ on $V$ is traceless.
Since $B$ acts through the two-dimensional $\D$-flavor factor of $V$ while
acting as the identity on the $q$-dimensional multiplicity space $W$, a
direct computation gives, for $B,C\in E_{\D}$,
$$
\tr_{\R}(A_BA_C)=q\,\Rea\tr(BC)
=q\langle B,C\rangle_E.
$$
The standard fourth-moment formula on the unit sphere in a real
$N$-dimensional Euclidean space gives
$$
\int (x^TAx)(x^TCx)\,d\sigma(x)
=
\frac{2\tr(AC)+\tr(A)\tr(C)}{N(N+2)}.
$$
Here $N=2q$ and both operators are traceless, so
$$
\int q_Bq_C\,d\sigma
=
\frac{\langle B,C\rangle_E}{2(q+1)}.
$$
This is the covariance identity.  The mean is zero by spherical symmetry.
\end{proof}

\begin{theorem}[Quadratic-defect completion]
\label{thm:completion}
Assume
$$
\Pfour(G,W)=0.
$$
Let
$$
\mu=\sum_{i=1}^s w_i\mu_{Gx_i},
\qquad
w_i>0,\qquad
\sum_iw_i=1,
$$
and put $Z_i=Z(x_i)$.  Then the following are equivalent:
\begin{enumerate}[label=\textup{(\roman*)}]
\item $\mu$ has the same averages as the sphere for every even polynomial of
degree at most $4$;
\item
$$
\sum_iw_iZ_i=0
$$
and
$$
\sum_iw_iZ_i\otimes Z_i
=
\frac{1}{2(q+1)}I_{E_{\D}}.
$$
\end{enumerate}
Consequently, whenever these equations hold, $\mu^{\pm}$ is a weighted
spherical $4$-design.  If $\mu$ already has vanishing cubic moments---for
example, if $\mu$ is antipodal or $\Harm_3(V)^G=0$---then $\mu$ itself is a
weighted spherical $4$-design.
\end{theorem}

\begin{proof}
Because $\mu$ is $G$-invariant, it is enough to test invariant polynomials.
There are no invariant linear forms because $W$ is nontrivial and irreducible.
The invariant harmonic quadratics are exactly
$$
q_B(x)=\langle B,Z(x)\rangle_E,
\qquad B\in E_{\D},
$$
so the degree-two condition is equivalent to $\sum_iw_iZ_i=0$.

Now restrict to the unit sphere.  There
$$
H(x)=\frac12I_2+Z(x).
$$
For $B\in\Herm_2(\D)$, write
$$
B_0:=B-\frac{\tr B}{2}I_2\in E_{\D}.
$$
Then every quadratic Gram coordinate has the form
$$
\gamma_B(x)
=\frac{\tr B}{2}+\langle B_0,Z(x)\rangle_E.
$$
Consequently every product $\gamma_B\gamma_C$, when restricted to the
sphere, is a sum of a constant term, a term linear in $Z$, and a term of the
form
$$
\langle B_0,Z\rangle_E\,\langle C_0,Z\rangle_E.
$$
Since $\Pfour(G,W)=0$, these products span all invariant quartics.  The
constant terms already have the correct average, the linear terms are
exactly the degree-two conditions just imposed, and the remaining quadratic
terms in $Z$ are matched for every $B_0,C_0\in E_{\D}$ precisely when
$$
\sum_iw_iZ_i\otimes Z_i
=
\frac{1}{2(q+1)}I_{E_{\D}},
$$
by Lemma~\ref{lem:cov}.  This proves the equivalence through degree four.

Antipodalization preserves the two even equations and kills all odd
polynomials.  The last assertion follows in the same way.
\end{proof}

\section{Optimal completions and projective defect designs}
\label{sec:optimal}

We now assume that the defect equations of
Theorem~\ref{thm:completion} hold and that at least one constituent orbit is
itself a spherical $2$-design.  By Proposition~\ref{prop:2design}, such an
orbit has defect zero.  Its Gram matrix is $I_2/2$ and therefore has rank
$2$, so necessarily $n\ge2$.

For any $x\in S(V)$, the two eigenvalues of $H(x)$ are nonnegative and sum to
one.

\begin{lemma}
\label{lem:maxdefect}
For every $x\in S(V)$,
$$
\|Z(x)\|^2\le\frac12.
$$
Equality holds if and only if
$$
\rank H(x)=1.
$$
\end{lemma}

\begin{proof}
If the eigenvalues of $H(x)$ are $\lambda$ and $1-\lambda$, then
$$
\|Z(x)\|^2
=
2\left(\lambda-\frac12\right)^2
\le\frac12.
$$
Equality holds exactly when $\lambda\in\{0,1\}$.
\end{proof}

\begin{remark}[The defect body]
\label{rem:defectbody}
Because $n\ge2$, every positive semidefinite matrix
$H\in\Herm_2(\D)$ of trace one occurs as the Gram matrix of two vectors in
$W$.  Hence Lemma~\ref{lem:maxdefect} describes the entire image of the
defect map:
$$
Z(S(V))
=
\left\{Z\in E_{\D}:\|Z\|^2\le\frac12\right\}.
$$
Thus the defect body is a Euclidean ball, and its boundary consists exactly
of the rank-one Gram matrices.  The sharpness statement below can therefore
be read geometrically: minimum correction weight forces every nonzero defect
to the boundary of this ball.
\end{remark}

Let
$$
\alpha:=\sum_{Z_i\ne0}w_i
$$
be the total weight carried by correction orbits.

\begin{theorem}[Optimal correction bound]
\label{thm:bound}
Under the defect equations,
$$
\alpha\ge\frac{r_{\D}}{q+1}
=
\frac{d+1}{dn+1}.
$$
Equivalently, the total weight carried by $2$-design orbits is at most
$$
\frac{d(n-1)}{dn+1}.
$$
Equality holds if and only if every correction orbit has rank-one Gram matrix.
\end{theorem}

\begin{proof}
Taking traces in the covariance equation gives
$$
\sum_iw_i\|Z_i\|^2
=
\frac{r_{\D}}{2(q+1)}.
$$
By Lemma~\ref{lem:maxdefect},
$$
\frac{r_{\D}}{2(q+1)}
\le
\frac{\alpha}{2}.
$$
Equality holds precisely when every nonzero defect has norm squared $1/2$.
\end{proof}

We call a completion \emph{correction-weight optimal} if equality holds in
Theorem~\ref{thm:bound}.

A rank-one positive semidefinite Hermitian matrix of trace one has the form
$$
P_u=uu^*,
\qquad
u\in\D^2,\quad\|u\|=1,
$$
and depends only on the projective line $[u]\in\D\mathbb P^1$.  The map
$$
[u]\longmapsto
\sqrt2\left(P_u-\frac12I_2\right)
$$
identifies $\D\mathbb P^1$ with the unit sphere of $E_{\D}$.

\begin{theorem}[Projective-line equality theorem]
\label{thm:projective}
Assume equality in Theorem~\ref{thm:bound}.  For each correction orbit write
$$
v_i:=\sqrt2Z_i\in S(E_{\D}),
\qquad
\lambda_i:=\frac{w_i}{\alpha}.
$$
Then
$$
\sum_i\lambda_iv_i=0,
\qquad
\sum_i\lambda_iv_i\otimes v_i
=
\frac1{r_{\D}}I_{E_{\D}}.
$$
Thus the normalized defects form a weighted spherical $2$-design on
$$
S(E_{\D})\simeq\D\mathbb P^1.
$$
Conversely, fix any zero-defect orbit and any weighted spherical
$2$-design $\{(v_i,\lambda_i)\}$ on $\D\mathbb P^1\simeq S(E_{\D})$.
Choose rank-one correction representatives with $Z_i=v_i/\sqrt2$, give them
weights
$$
w_i=\alpha\lambda_i,
\qquad
\alpha=\frac{r_{\D}}{q+1},
$$
and place the remaining mass $1-\alpha$ on the zero-defect orbit.  Then the
defect equations hold with equality in Theorem~\ref{thm:bound}.  Hence,
after antipodalization if necessary, one obtains a correction-weight optimal
weighted spherical $4$-design.
\end{theorem}

\begin{corollary}[Minimum number at optimal correction weight]
A correction-weight optimal completion requires at least
$$
r_{\D}+1=d+2
$$
correction orbits.  If this minimum is attained, the correction weights are
necessarily equal and the normalized defects form a regular simplex.  Thus the
minimal correction configurations are
$$
\begin{array}{c|c|c}
\D&\D\mathbb P^1&\text{minimal correction design}\\
\hline
\R&S^1&\text{equilateral triangle}\\
\C&S^2&\text{regular tetrahedron}\\
\HH&S^4&\text{regular $5$-simplex}.
\end{array}
$$
\end{corollary}

\begin{proof}
Let $r=r_{\D}$.  The covariance equation implies that the correction defects
span $E_{\D}$, so there are at least $r$ of them.  There cannot be exactly
$r$: in that case the spanning defects would be linearly independent, while
the first-moment equation
$$
\sum_i\lambda_i v_i=0
$$
is a nontrivial linear relation because every $\lambda_i$ is positive.
Hence at least $r+1$ correction orbits are required.

Suppose now that this minimum is attained.  Set
$u_i=\sqrt{r\lambda_i}\,v_i$ and let $U$ be the $r\times(r+1)$ matrix with
columns $u_i$.  The second-moment equation gives $UU^T=I_r$, while the
first-moment equation gives $U(\sqrt{\lambda_i})_i=0$.  Since $U$ has rank
$r$, its Gram matrix is
$$
U^TU=I_{r+1}-aa^T,
\qquad a=(\sqrt{\lambda_i})_i.
$$
Comparing diagonal entries gives $r\lambda_i=1-\lambda_i$, so
$\lambda_i=1/(r+1)$ for every $i$.  The off-diagonal entries then give the
regular-simplex inner products.
\end{proof}

For the minimal simplex, the base and correction weights are
$$
w_0=\frac{d(n-1)}{dn+1},
\qquad
w_c=\frac{d+1}{(dn+1)(d+2)}.
$$

The Clifford family gives a uniform illustration in growing dimension.

\begin{corollary}[Clifford completions in unbounded dimension]
\label{cor:cliffordcompletion}
Let $r\ge1$, let
$$
G_r=\widetilde{\mathrm{Cliff}}_r,
\qquad
W_r=(\C^{2^r})_{\R},
\qquad
V_r=W_r\oplus W_r,
$$
and choose two orthonormal vectors $e_1,e_2\in\C^{2^r}$.  Put
$$
M_0=\frac1{\sqrt2}[e_1\ e_2].
$$
Let $[u_1],\ldots,[u_4]\in\C\mathbb P^1$ be the vertices of a regular
tetrahedron under the identification
$\C\mathbb P^1\simeq S(E_{\C})\simeq S^2$.  For any unit vector
$w\in W_r$, set
$$
C_j=w\,u_j^*,
\qquad 1\le j\le4.
$$
Then the weighted $G_r$-invariant measure
$$
\mu_r=
\frac{2(2^r-1)}{2^{r+1}+1}\,\mu_{G_rM_0}
+
\frac{3}{4(2^{r+1}+1)}
\sum_{j=1}^4\mu_{G_rC_j}
$$
is a weighted spherical $5$-design attaining equality in
Theorem~\ref{thm:bound} on
$$
S(V_r)=S^{2^{r+2}-1}.
$$
In particular, the ambient dimension tends to infinity with $r$, while the
correction-weight optimal completion with the minimum number of correction
orbits is governed by the same regular tetrahedron
in the fixed three-dimensional defect space $E_{\C}$.
\end{corollary}

\begin{proof}
By Corollary~\ref{cor:cliffordfamily}, the primitive quartic obstruction
vanishes and $\D=\C$.  Since $M_0^*M_0=I_2/2$, the base orbit has zero
defect.  Each $C_j$ has rank-one Gram matrix $u_ju_j^*$, and the four
normalized defects form a regular tetrahedron, hence an equal-weight
spherical $2$-design on $S(E_{\C})$.  The displayed coefficients are the
minimal-simplex weights with $d=2$ and $n=2^r$.  Theorems
\ref{thm:completion} and~\ref{thm:projective} therefore give a correction-weight optimal weighted spherical $4$-design.  Finally,
$-I=(iI)^2\in\mathcal P_r\subseteq G_r$, so every constituent orbit is
antipodal and all odd moments vanish; the weighted design therefore has
strength $5$.
\end{proof}

\section{Three model examples}
\label{sec:examples}

We now give one model example in each Frobenius--Schur type.  The table
records exactly the parameters used by the defect criterion: the real
ambient dimension $q$, the defect dimension $r_{\D}$, and the cubic and
quartic invariant multiplicities.  The real-type example is unweighted,
while the complex- and quaternionic-type examples are weighted.

$$
\begin{array}{c|c|c|c|c|c}
\D&G&q&r_{\D}&m_3(W)&m_4(W)\\
\hline
\R&A_5&3&2&0&6\\
\C&\widetilde{\mathrm{Cliff}}_1&4&3&0&10\\
\HH&H_{720}\cong2.A_6&8&5&0&21.
\end{array}
$$

\subsection{Real type: the rotational icosahedral group}

Let $G=A_5$ act on its natural $3$-dimensional real rotation module
$W=\R^3$.  Here
$$
\D=\R,\qquad n=3,\qquad q=3,\qquad r_{\D}=2.
$$
The character criterion gives $m_4(W)=6$, so there is no primitive quartic
obstruction; moreover $\Harm_3(W\oplus W)^G=0$.

Choose an orthonormal frame $(a,b)$ with trivial stabilizer in $A_5$ and set
$$
M_0=\frac1{\sqrt2}[a\ b].
$$
Then $GM_0$ is a spherical $2$-design with $60$ points.  Such frames exist
generically because $A_5$ is finite.

Choose a unit vector $w$ on a $3$-fold rotation axis, so $|Gw|=20$.  Let
$$
u_j=
\begin{pmatrix}
\cos(j\pi/3)\\
\sin(j\pi/3)
\end{pmatrix},
\qquad j=0,1,2,
$$
and put
$$
C_j=w\,u_j^T.
$$
The three normalized defects form an equilateral triangle in
$\R\mathbb P^1\simeq S^1$.  Their Gram matrices $u_ju_j^T$ are distinct, so
the three correction orbits are distinct; they are also disjoint from the
rank-two base orbit.  Since
$$
\frac{60}{60+3\cdot20}=\frac12
=
\frac{n-1}{n+1},
$$
the ordinary union
$$
GM_0\cup GC_0\cup GC_1\cup GC_2
$$
is an unweighted $120$-point spherical $4$-design in $S^5$.

\subsection{Complex type: the one-qubit Clifford group}

This is the case $r=1$ of Corollary~\ref{cor:cliffordcompletion}.  Thus
$$
G=\widetilde{\mathrm{Cliff}}_1=\langle \mathsf H,S\rangle\subset U(2)
$$
is the standard finite one-qubit Clifford group of order $192$, with the
matrices $\mathsf H$ and $S$ defined in Section~\ref{sec:primitive}.  For its
realification $W$,
$$
\D=\C,\qquad n=2,\qquad q=4,\qquad r_{\D}=3.
$$
We record the construction explicitly because it displays the tetrahedral
defect geometry in the smallest complex-type case.

Take
$$
M_0=\frac1{\sqrt2}I_2.
$$
Then $GM_0$ is a spherical $2$-design.

For the four correction directions, take the tetrahedral lines in $\C^2$
represented by
$$
u_1=
\begin{pmatrix}1\\0\end{pmatrix},
\qquad
u_2=
\begin{pmatrix}
1/\sqrt3\\
\sqrt{2/3}
\end{pmatrix},
$$
$$
u_3=
\begin{pmatrix}
1/\sqrt3\\
\sqrt{2/3}\,e^{2\pi i/3}
\end{pmatrix},
\qquad
u_4=
\begin{pmatrix}
1/\sqrt3\\
\sqrt{2/3}\,e^{4\pi i/3}
\end{pmatrix}.
$$
They satisfy
$$
|\langle u_i,u_j\rangle|^2=\frac13
\qquad(i\ne j),
$$
so their centered projectors form a regular tetrahedron in
$\C\mathbb P^1\simeq S^2$.  For any unit vector $w\in W$, put
$$
C_j=w\,u_j^*.
$$
The weighted union
$$
\frac25\,\mu_{GM_0}
+
\frac{3}{20}\sum_{j=1}^4\mu_{GC_j}
$$
is a weighted spherical $4$-design in $S^7$.  As in
Corollary~\ref{cor:cliffordcompletion} ($-I=(iI)^2\in G$), the measure is
antipodal and hence is automatically a weighted spherical $5$-design.

\subsection{Quaternionic type: the primitive group \texorpdfstring{$2.A_6$}{2.A6}}

Let
$$
G=H_{720}\cong2.A_6
$$
be the primitive quaternionic reflection group of order $720$ acting
irreducibly on $W=\HH^2$.  This realization and its six distinguished
equiangular quaternionic lines are described in \cite{Waldron2024}.  Here
$$
\D=\HH,\qquad n=2,\qquad q=8,\qquad r_{\D}=5.
$$
Let $\chi_8$ denote the faithful degree-$4$ complex character of $2.A_6$;
it has Frobenius--Schur indicator $-1$ and Schur index $2$
\cite{Atlas1985}.  Since $W$ is the associated $8$-dimensional real module,
its real character is $2\chi_8$.
Substituting this character into Proposition~\ref{prop:character} and
\eqref{eq:cubic-character} gives
$$
m_4(W)=21,
\qquad
m_3(W)=0.
$$
Thus there is no primitive quartic obstruction and no cubic obstruction.

Take
$$
M_0=\frac1{\sqrt2}I_2.
$$
Then $GM_0$ is a spherical $2$-design.

Let
$$
[u_1],\ldots,[u_6]\in\HH\mathbb P^1
$$
be the six equiangular quaternionic lines.  They satisfy
$$
|\langle u_i,u_j\rangle|^2=\frac25
\qquad(i\ne j).
$$
Consequently
$$
Z_j=u_ju_j^*-\frac12I_2
$$
satisfy
$$
\|Z_j\|^2=\frac12,
\qquad
\langle Z_i,Z_j\rangle=-\frac1{10}
\quad(i\ne j).
$$
Thus the normalized defects $\sqrt2Z_j$ form a regular $5$-simplex in
$\HH\mathbb P^1\simeq S^4$.

For any unit $w\in W$, put
$$
C_j=w\,u_j^*.
$$
The weighted union
$$
\frac49\,\mu_{GM_0}
+
\frac{5}{54}\sum_{j=1}^6\mu_{GC_j}
$$
is a weighted spherical $4$-design in $S^{15}$.  The central involution of
the non-split extension $2.A_6$ acts as a scalar quaternion on the faithful
irreducible $W$ (Schur's lemma).  Since it has order two, that scalar is
$\pm1$, and faithfulness excludes $+1$.  Thus $-I\in G$, so the measure is
automatically a weighted spherical $5$-design.

\section{Unweighted completions}
\label{sec:unweighted}

Suppose $GM_0$ is a $2$-design orbit of size $N_0$.  Then $n\ge2$, as noted
in Section~\ref{sec:optimal}.  Suppose further that $k$ rank-one correction
orbits have a common size $N_1$ and equal defect weights.  If the
normalized defects form an equal-weight spherical $2$-design on
$S(E_{\D})$, then the union solves the even moment equations precisely when
$$
\frac{kN_1}{N_0+kN_1}
=
\frac{d+1}{dn+1}.
$$
Equivalently,
$$
k=
\frac{(d+1)N_0}{d(n-1)N_1},
$$
and in that case
$$
N_0+kN_1
=
\frac{dn+1}{d(n-1)}N_0.
$$

The common-size and equal-weight assumptions are used only for the preceding
formula.  Independently of them, Theorem~\ref{thm:bound} immediately gives
the following general bound.

\begin{corollary}
Any unweighted completion satisfying the even fourth-moment equations and
containing a prescribed $2$-design orbit of size $N_0$ satisfies
$$
|X|\ge
\frac{dn+1}{d(n-1)}N_0.
$$
If equality holds, every correction orbit has rank-one Gram matrix.
\end{corollary}

If the resulting configuration has vanishing cubic moments, it is itself a
spherical $4$-design; otherwise its antipodalization is one.

\section{A sharp \texorpdfstring{$E_6$}{E6} completion}
\label{sec:E6}

Let $G=W(E_6)$ act on its $6$-dimensional real reflection representation
$W$, and put $V=W\oplus W$.  Applying the character criterion to the
standard character data of $W(E_6)$ \cite{GeckPfeiffer2000} gives
\[
m_4(W)=6,
\]
so there is no primitive quartic obstruction.  Equation~\eqref{eq:cubic-character} also gives
\[
m_3(W)=0,
\]
equivalently $\Harm_3(V)^G=0$.
Consequently, any solution of the defect equations is already a spherical
$4$-design, without antipodalization.

We use the following explicit simple-root coordinates.  Number the simple
roots so that the Cartan matrix is
\[
C=
\begin{pmatrix}
2&0&-1&0&0&0\\
0&2&0&-1&0&0\\
-1&0&2&-1&0&0\\
0&-1&-1&2&-1&0\\
0&0&0&-1&2&-1\\
0&0&0&0&-1&2
\end{pmatrix}.
\]
The simple roots are normalized to have squared length $2$, so the Euclidean
inner product in these coordinates is
\[
(v,w)=v^TCw.
\]
In these coordinates, the simple reflection $s_i$ acts on a vector $v$ by
\[
s_i(v)=v-(Cv)_ie_i,
\]
where $e_i$ is the $i$th coordinate vector; in particular, $s_i$ fixes $v$
exactly when $(Cv)_i=0$.  Set
\[
a_0=(2,3,4,6,5,4)^T,
\qquad
b_0=(2,1,2,2,1,0)^T.
\]
Then
\[
Ca_0=(0,0,0,0,0,3)^T,
\qquad
Cb_0=(2,0,0,0,0,-1)^T,
\]
and
\[
(a_0,a_0)=12,\qquad
(b_0,b_0)=4,\qquad
(a_0,b_0)=0.
\]
Thus
\[
a=\frac{a_0}{\sqrt{12}},
\qquad
b=\frac{b_0}{2}
\]
are orthonormal.  Moreover, $Ca_0$ has zero entries exactly on nodes
$1,\ldots,5$ and a positive sixth entry.  Hence $a$ lies in the relative
interior of the corresponding face of the closed fundamental chamber.  By
the standard face-stabilizer description for finite reflection groups
\cite{Humphreys1990}, its full stabilizer is the $D_5$ parabolic
$$
\operatorname{Stab}_G(a)=\langle s_1,\ldots,s_5\rangle\cong W(D_5).
$$
Within this $D_5$ subgroup, the simple-root pairings of $b_0$ on nodes
$1,\ldots,5$ are $(2,0,0,0,0)$.  Thus the full stabilizer of the ordered
pair $(a,b)$ is
$$
\operatorname{Stab}_G(a,b)=\langle s_2,s_3,s_4,s_5\rangle\cong W(D_4).
$$

Set
\[
M_0=\frac1{\sqrt2}[a\ b].
\]
Then
\[
M_0^TM_0=\frac12I_2,
\]
so $GM_0$ is a spherical $2$-design.  By the stabilizer computation above
(and standard parabolic stabilizer theory \cite{Humphreys1990}),
\[
|GM_0|
=
\frac{|W(E_6)|}{|W(D_4)|}
=
\frac{51840}{192}
=
270.
\]

For the correction orbits, put
\[
u_j=
\begin{pmatrix}
\cos(j\pi/4)\\
\sin(j\pi/4)
\end{pmatrix},
\qquad j=0,1,2,3,
\]
and
\[
C_j=a\,u_j^T.
\]
Since $u_j\ne0$, the stabilizer of $C_j=a\,u_j^T$ is exactly the stabilizer
of $a$, namely $W(D_5)$.  Therefore
\[
|GC_j|=|Ga|
=
\frac{51840}{1920}
=
27.
\]
The orbit $Ga\subset S(W)=S^5$ is the classical $27$-point Schläfli tight
spherical $4$-design associated with $E_6$ \cite{HardinSloane1992}.  Hence
each $GC_j$ is an isometric copy of this configuration inside the
six-dimensional subspace $W\otimes u_j\subset V$.  In the ambient sphere
$S^{11}$, however, a rank-one correction orbit is not even a spherical
$2$-design; its role here is instead measured by its nonzero defect.

The four Gram matrices $C_j^TC_j=u_ju_j^T$ are distinct, so the four
correction orbits are distinct; they are also disjoint from the rank-two base
orbit.  Their normalized defects form a square in
$\R\mathbb P^1\simeq S^1$.

\begin{theorem}
\label{thm:E6construction}
The set
\[
X=
GM_0\sqcup GC_0\sqcup GC_1\sqcup GC_2\sqcup GC_3
\subset S^{11}
\]
is an unweighted spherical $4$-design with
\[
|X|=270+4\cdot27=378.
\]
\end{theorem}

\begin{proof}
The natural orbit weights are
\[
\frac{270}{378}=\frac57,
\qquad
\frac{27}{378}=\frac1{14}.
\]
The base orbit has defect zero.  The four correction defects have maximal
norm and, after normalization, form the equal-weight spherical $2$-design
given by a square on $S^1$.  Theorem~\ref{thm:projective} therefore gives the
required second and fourth moments.  Since $\Harm_3(V)^G=0$,
Theorem~\ref{thm:completion} shows that $X$ itself is a spherical $4$-design.
\end{proof}

We next show that the size in Theorem~\ref{thm:E6construction} is sharp
within the completion class considered in this paper.

\begin{theorem}[Sharpness and equality structure for $E_6$ completions]
\label{thm:E6optimal}
Let $Y\subset S^{11}$ be an unweighted $W(E_6)$-invariant spherical
$4$-design containing a $W(E_6)$-orbit which is itself a spherical
$2$-design.  Then
\[
|Y|\ge378.
\]
If equality holds, then $Y$ consists of exactly one $270$-point $2$-design
orbit and four $27$-point rank-one correction orbits.  The four normalized
quadratic defects form a square in $S^1$.
\end{theorem}

\begin{proof}
Let $GM$ be a constituent orbit which is itself a spherical $2$-design.
Then $M^TM=I_2/2$, so the two columns of $M$ are linearly independent.
Its stabilizer therefore fixes a $2$-dimensional subspace of $W$ pointwise.
Pointwise stabilizers in a finite real reflection group are parabolic
\cite{Humphreys1990}.  Among parabolic subgroups of $W(E_6)$ of rank at most
$4$, the largest is $W(D_4)$, of order $192$.  Hence
\[
|GM|\ge \frac{51840}{192}=270.
\]

For real type with $n=6$, Theorem~\ref{thm:bound} says that the total weight
of all zero-defect orbits is at most
\[
\frac{n-1}{n+1}=\frac57.
\]
Since $Y$ is unweighted and contains an orbit of size at least $270$,
\[
\frac{270}{|Y|}\le\frac57,
\]
which gives $|Y|\ge378$.

Suppose now that $|Y|=378$.  Equality must hold throughout the preceding
argument.  Thus there is exactly one zero-defect orbit, it has size $270$,
and every correction orbit has rank-one Gram matrix.  A rank-one
representative has the form $wu^T$, and its stabilizer is the stabilizer of
the nonzero vector $w\in W$.  Such a stabilizer is again parabolic.  There are
$108$ correction points.  Any correction orbit occurring here therefore has
size at most $108$, so its stabilizer has order at least $480$.  The maximal parabolic types are
$D_5$, $A_5$, $A_1A_4$, and $A_1A_2A_2$ (with repetitions according to the
deleted node); see also the $E_6$ parabolic table in
\cite[Table~12]{KrishnasamyTaylor2018}.  Their orders show that the only
maximal parabolics of order at least $480$ have types $D_5$ and $A_5$, giving
orbit sizes
\[
27\quad\text{and}\quad72,
\]
respectively; lower-rank parabolics give larger orbits.  Hence, if $r$ and
$s$ are the numbers of correction orbits of sizes $27$ and $72$, then
\[
27r+72s=108.
\]
The only nonnegative integral solution is
\[
r=4,\qquad s=0.
\]

The four correction orbits consequently have equal weights.  At equality in
Theorem~\ref{thm:bound}, their normalized defects are an equal-weight
spherical $2$-design of four points on $S^1$.  Such a design is necessarily a
square.  Indeed, viewing the four defects as unit complex numbers $z_i$, the
condition $\sum_i z_i=0$ also gives
$e_3=(\prod_i z_i)\overline{\sum_i z_i}=0$, so the monic polynomial with
roots $z_i$ is even and the points occur in two antipodal pairs.  The isotropic
second moment then makes the two pair directions orthogonal.
\end{proof}

\section*{Statement on the use of artificial intelligence}

OpenAI's ChatGPT was used during the preparation of this manuscript to assist
with exploratory mathematical reasoning and computations, literature searches,
the generation and checking of verification code, and the revision of
exposition.  The mathematical arguments, computations, final statements, and
cited references remain the responsibility of the authors.

\begingroup
\small
\bibliographystyle{plain}
\bibliography{quadratic_defect_completions_revised_v7}
\endgroup

\end{document}